\documentclass[a4paper, reqno, 14pt]{amsart}
\usepackage[utf8]{inputenc}

\usepackage{amsthm,amsfonts,amssymb,amsmath}
\usepackage{xr-hyper}
\usepackage[colorlinks=true]{hyperref}
\usepackage{verbatim}
\usepackage{pdfsync}

\usepackage{latexsym}
\usepackage{amssymb}
\usepackage{amsmath}
\usepackage{amsthm}
\usepackage{verbatim}
\usepackage{pdfsync}
\usepackage[curve,matrix,arrow,frame,tips]{xy}

\begin{document}

\newcommand{\now}{\count0=\time 
\divide\count0 by 60
\count1=\count0
\multiply\count1 by 60
\count2= \time
\advance\count2 by -\count1
\the\count0:\the\count2}





\newcommand{\C}{{\mathbb C}}
\newcommand{\Sym}{{\text{\rm Sym}}}
\newcommand{\Q}{{\mathbb Q}}

\font\cute=cmitt10 at 12pt
\font\smallcute=cmitt10 at 9pt
\newcommand{\kay}{{\text{\cute k}}}
\newcommand{\smallkay}{{\text{\smallcute k}}}

\newcommand{\Cal}{\mathcal}
\newcommand{\M}{\Cal M}

\newcommand{\End}{\text{\rm End}}
\newcommand{\Spec}{\text{\rm Spec}\, }
\newcommand{\F}{{\mathbb F}}
\newcommand{\R}{{\mathbb R}}
\newcommand{\Hom}{\text{\rm Hom}}
\newcommand{\ZZ}{\Cal Z}
\newcommand{\Z}{{\mathbb Z}}
\newcommand{\hfb}{\hfill\break}
\renewcommand{\o}{\omega}
\newcommand{\ph}{\varphi}
\newcommand{\GL}{\text{\rm GL}}


\newtheorem{theo}{Theorem}[section]
\newtheorem{rem}[theo]{Remark}
\newtheorem{lem}[theo]{Lemma}
\newtheorem{conj}[theo]{Conjecture}
\renewcommand{\theequation}{\thesection.\arabic{equation}}
\numberwithin{equation}{section}


\newcommand{\E}{{\mathbb E}}

\newcommand{\OO}{\text{\rm O}}
\newcommand{\UU}{\text{\rm U}}

\newcommand{\OK}{O_{\smallkay}}
\newcommand{\DI}{\mathcal D^{-1}}

\newcommand{\pre}{\text{\rm pre}}

\newcommand{\Bor}{\text{\rm Bor}}
\newcommand{\Rel}{\text{\rm Rel}}
\newcommand{\rel}{\text{\rm rel}}
\newcommand{\Res}{\text{\rm Res}}
\newcommand{\TG}{\widetilde{G}}

\parindent=0pt
\parskip=6pt
\baselineskip=14pt

\newcommand{\PP}{\mathcal P}
\renewcommand{\OO}{\mathcal O}
\newcommand{\BB}{\mathbb B}
\newcommand{\GU}{\text{\rm GU}}
\newcommand{\Herm}{\text{\rm Herm}}

\newcommand{\FF}{\mathbb F}
\newcommand{\MM}{\mathbb M}
\newcommand{\YY}{\mathbb Y}
\newcommand{\VV}{\mathbb V}
\newcommand{\LL}{\mathbb L}

\newcommand{\subover}[1]{\overset{#1}{\subset}}
\newcommand{\supover}[1]{\overset{#1}{\supset}}

\newcommand{\hgs}[2]{\{#1,#2\}}

\newcommand{\wh}[1]{\widehat{#1}}

\newcommand{\Ker}{\text{\rm Ker}}
\newcommand{\YYbar}{\overline{\YY}}
\newcommand{\MMbar}{\overline{\MM}}
\newcommand{\oY}{\overline{Y}}
\newcommand{\dra}{\dashrightarrow}
\newcommand{\dlra}{\longdashrightarrow}

\newcommand{\yy}{\text{\bf y}}

\newcommand{\red}{\text{\rm red}}

\newcommand{\inc}{\text{\rm inc}}

\newcommand{\OKs}{O_{\smallkay,s}}
\newcommand{\OKr}{O_{\smallkay,r}}
\newcommand{\Xs}{X^{(s)}}
\newcommand{\Xo}{X^{(0)}}

\newcommand{\xss}{\underset{\sim}{x}}
\newcommand{\xxss}{\underset{\sim}{\xx}}
\newcommand{\yss}{\underset{\sim}{y}}

\renewcommand{\ss}{\text{\rm ss}}

\newcommand{\OKp}{O_{\smallkay,p}}

\newcommand{\cutter}{\vskip .14in\hrule\vskip .04in}

\newcommand{\I}{\mathbb I}

\newcommand{\hh}{\mathtt h}
\newcommand{\Nilp}{\text{\rm Nilp}}
\newcommand{\nai}{\text{\rm naive}}

\newcommand{\nut}{\widetilde{\nu}}
\newcommand{\oht}{\widehat{O}^\times}
\newcommand{\oh}{\widehat{O}}
\newcommand{\zh}{\widehat{\Z}}
\newcommand{\zht}{\zh^\times}
\newcommand{\ktaf}{\kay^\times_{\A_f}}
\newcommand{\aaa}{\frak a}


\newcommand{\Lie}{\text{\rm Lie}\, }
\newcommand{\CK}{C_{\smallkay}}

\newcommand{\N}{\Cal N}
\newcommand{\uuxx}{\und{\und{\xx}}}

\newcommand{\OKt}{\widetilde{O}_\smallkay}
\newcommand{\Lt}{\widetilde{L}}

\newcommand{\sra}{\rightarrow}

\newcommand{\MHr}{\M}   
\newcommand{\MH}{\M}

\newcommand{\Sh}{\text{\rm Sh}}
\newcommand{\SSh}{\und{\Sh}}

\renewcommand{\aa}{a}

\newcommand{\LLL}{[[L]]}

\setcounter{tocdepth}{1}

\title{On occult period maps}

\author{Stephen Kudla
\medskip\\
and
\medskip\\
Michael Rapoport}

\maketitle

\centerline{\it \hfill In memoriam  Jonathan Rogawski}

\bigskip
\centerline{\bf Abstract} 
We interpret the``occult" period maps of Allcock, Carlson, Toledo \cite{ACT1, ACT2}, resp. of 
Looijenga, Swierstra \cite{LS1, LS2}, resp. of Kondo \cite{K1, K2} in moduli theoretic terms, 
as a construction of certain families of polarized abelian varieties of Picard type. We show that 
these period maps are morphisms defined over their natural field of definition. 

\section{Introduction}\label{introduction}

In papers of Allcock, Carlson, Toledo \cite{ACT1, ACT2}, resp. of Looijenga, Swierstra \cite{LS1, LS2}, resp. of 
Kondo \cite{K1, K2},  "hidden" period maps are constructed in certain  cases. The target spaces of these maps 
are certain {\it arithmetic quotients of complex unit balls}. The basic 
observation, which is the starting  point of this paper, is that these arithmetic quotients can be interpreted as the complex points of certain {\it moduli spaces of abelian varieties 
of Picard type}, of the kind  considered in our paper \cite{KR2}. Consequently, the purpose in this paper is to interpret these hidden period maps in moduli-theoretic terms.  
The pay-off of this exercise is that we can raise and partially answer some {\it descent problems} which seem natural from our view point, and which are related to a  similar 
descent problem addressed by Deligne in \cite{De1} in his theory of {\it complete intersections of Hodge level one}.

Why do we speak of "hidden", or "occult" period maps in this context? This is done in order to make the distinction 
with the usual period maps which associate to a family of smooth projective complex varieties (over some 
base scheme $S$)  the (polarized) Hodge structures of its fibers, which 
then induces a map from $S$ to a quotient by a discrete group of a period domain. Let us recall three examples of classical period maps:

(1) {\bf Case of quartic surfaces.} In this case the period map is a holomorphic map of orbifolds 
\begin{equation*}
\varphi: {\it Quartics}_{2, \C}^\circ\to \big[\Gamma\backslash V(2, 19)\big].
\end{equation*}
\noindent{\rm Here ${\it Quartics}_{2, \C}^\circ$ denotes the stack parametrizing smooth quartic surfaces up to projective 
equivalence,
\begin{equation*}
\text{\it Quartics}_{2, \C}^\circ = \big[{\rm PGL}_4\backslash \mathbb P\Sym^4 (\C^4)^\circ \big]
\end{equation*}
{\rm (stack quotient in the orbifold sense).} The target space is the orbifold quotient of the space of 
oriented positive $2$-planes in a quadratic space $V$ of signature $(2, 19)$ by the automorphism group $\Gamma$ of a lattice in $V$.  

(2) {\bf Case of cubic threefolds.} In this case the period map is a holomorphic map of orbifolds 
\begin{equation*}
\varphi: {\it Cubics}_{3, \C}^\circ\to \big[\Gamma\backslash\mathfrak H_5\big].
\end{equation*}
\noindent{\rm Here ${\it Cubics}_{3, \C}^\circ$ denotes the stack parametrizing smooth cubic threefolds up to projective 
equivalence.
The target space is the orbifold quotient of the Siegel upper half space of genus $5$ by the Siegel group $\Gamma={\rm Sp}_5(\mathbb Z)$. 

(3) {\bf Case of cubic fourfolds.} In this case the period map is a holomorphic map of orbifolds 
\begin{equation*}
\varphi: {\it Cubics}_{4, \C}^\circ\to \big[\Gamma\backslash V(2, 20)\big].
\end{equation*}
\noindent{\rm Here ${\it Cubics}_{4, \C}^\circ$ denotes the stack parametrizing smooth cubic fourfolds up to projective 
equivalence.
The target space is the orbifold quotient of the space of oriented positive $2$-planes in a quadratic space 
$V$ of signature $(2, 20)$ by the automorphism group $\Gamma$ of a lattice in $V$.  

In the first case, by the Torelli theorem of Piatetskii-Shapiro/Shafarevich, the induced map $|\varphi|$ on coarse moduli spaces is an open embedding. In the second case, 
by the Torelli theorem of Clemens/Griffiths, the map $|\varphi|$ is a locally closed embedding  (it is not an open embedding since the 
source of $\varphi$ has dimension $10$, and the target has dimension $15$).   In the third case, by the Torelli 
theorem of Voisin, the map $|\varphi|$ is an open embedding.


The construction of the occult period maps is quite different, although it does use the classical period maps indirectly. 
For instance, the construction of Allcock, Carlson, Toledo attaches a certain Hodge structure to any smooth cubic 
surface which allows one to distinguish between non-isomorphic ones, even though the natural Hodge structures on the 
cohomology in the middle dimension of all cubic surfaces are isomorphic. 
Also, in one dimension higher, their construction allows them to define an {\it open embedding} of the space of cubic threefolds into 
an arithmetic quotient of the complex unit ball of dimension $10$.

Our second aim in this paper is to  identify the complements of the images of occult period maps 
with  {\it special divisors} considered in \cite{KR2}. 

The lay-out of the paper is as follows. In sections \ref{subsectionpic}, \ref{cu}  and \ref{sc} we recall 
some of the theory and notation of \cite{KR2}. In sections \ref{cubicsurfaces},  \ref{cubicthreefolds}, 
 \ref{curvesgenus3}, and \ref{curvesgenus4}, respectively, 
 we explain in turn the case of cubic surfaces,  cubic threefolds,  curves of genus $3$, and curves of genus $4$. 
 In section \ref{descent}, we explain the descent problem, and solve it in zero characteristic. In the final section, 
we make a few supplementary  remarks.

We stress that the proofs of our statements are all contained in the papers mentioned above, and that our work only consists in interpreting  these results.

We thank B.~van Geemen, D.~Huybrechts and E.~Looijenga for very helpful discussions. 
We also thank J.~Achter for keeping us informed about his progress in proving our conjecture in section \ref{descent} in some cases. Finally, we thank the referee who alerted us to a mistake concerning the stacks aspect of period maps.

\section{Moduli spaces of Picard type}\label{subsectionpic} 
Let $\kay=\Q(\sqrt{\Delta})$ be an imaginary-quadratic field with discriminant $\Delta$, ring of integers $\OK$, and a 
fixed complex embedding.   We write $\aa\mapsto{\aa^\sigma}$ for the
non-trivial automorphism of $\OK$.

For integers $n\geq 1$ and $r$, $0\leq r\leq n$, we consider the groupoid $\M = \M (n-r, r) = \M (\kay; n-r, r)$ fibered over $({\rm Sch} / \OK)$ which associates to an 
$\OK$-scheme $S$ the groupoid of triples $(A, \iota, \lambda)$. Here $A$ is an abelian scheme over $S$, 
$\lambda$ is a principal polarization, and $\iota : \OK\to\End (A)$ is a homomorphism such that
\begin{equation*}
\iota (\aa)^\ast = \iota ({\aa^\sigma})\ ,
\end{equation*}
for the Rosati involution $\ast$ corresponding to $\lambda$.   In addition, the following signature condition is imposed
\begin{equation}
{\rm char} (T, \iota (\aa)\mid\Lie A) = (T-i(\aa))^{n-r}\cdot (T-i({\aa^\sigma}))^r\ ,\quad\forall\,\aa\in\OK\ ,
\end{equation}
where $i:\OK\rightarrow \Cal O_S$ is the structure map. 

We will mostly consider the complex fiber $\M_\C = \M\times_{\Spec\OK}\Spec\C$ of $\M$. In any
case, $\M$ is a Deligne-Mumford stack and $\M_\C$ is smooth. We denote by $|\M_\C|$ the coarse moduli scheme. 

We will also have to consider the following variant,  defined  by modifying the requirement above that 
the polarization $\lambda$ be principal. Let $d>1$ be  a square free divisor of $\vert \Delta\vert$. 
Then $\M (\kay, d; n-r, r)^*=\M (\kay; n-r, r)^*$ parametrizes triples $(A, \iota, \lambda)$ as in the case 
of  $\M (\kay; n-r, r)$, except that we impose  the following condition on $\lambda$. We require first of all 
that $ \ker\lambda\subset A[d]$, so that $\OK/(d)$ acts on $\ker\lambda$. In addition, we require 
that this action factor  through the quotient ring $\prod_{p\mid d}\F_p$ of $\OK/(d)$, and that $\lambda$ 
be of degree $d^{n-1}$, if $n$ is odd, resp.\ $d^{n-2}$, if $n$ is even.  In the notation introduced in section 13 of \cite{KR2}, we have 
$\M (\kay, d; n-r, r)^*=\M (\kay, {\bf t}; n-r, r)^{*, {\rm naive}}$, where  the function ${\bf t}$ on the set of 
primes $p$ with  $p\mid \Delta$  assigns to $p$ the  integer $2[(n-1)/2]$ if $p\mid d$, and $0$ 
if $p\nmid d$. Note that if $\kay$ is the Gaussian field $\kay=\Q(\sqrt{-1})$, then necessarily $d=2$; if $\kay$ is the Eisenstein field  $\kay=\Q(\sqrt{-3})$, then $d=3$. We denote by $|\M_\C^*|$ the corresponding coarse moduli scheme. 

\section{Complex uniformization}\label{cu} 
Let us recall from \cite{KR2} the complex uniformization of $\M (\kay; n-1, 1)(\C)$ in the special case that $\kay$ has class number one. 
 For $n>2$, let $(V, (\ ,\  ))$ be a hermitian vector space over $\kay$ of signature
$(n-1, 1)$ which contains a self-dual $\OK$-lattice $L$. By the class number hypothesis, $V$ is unique up to isomorphism.  
When $n$ is odd,  or when $n$ is even and $\Delta$ is odd, the lattice $L$ is also unique up to isomorphism. We assume 
that one of these conditions is satisfied. 
 Let $\mathcal D$ be the space of negative lines in the
$\C$-vector space $(V_\R, \I_0)$, where the complex structure $\I_0$ is defined in terms of the discriminant 
of $\kay$, as $\I_0 = \sqrt{\Delta}/{|\sqrt{\Delta}|}$.
Let $\Gamma$  be the isometry group of $L$. 
Then the complex uniformization is the isomorphism of orbifolds, 
\begin{equation*}
\M (\kay; n-1, 1)(\C)\simeq [\Gamma\backslash \mathcal D] .
\end{equation*} 
There is an obvious $*$-variant of this uniformization, which gives 
\begin{equation*}
\M (\kay; n-1, 1)^*(\C)\simeq [\Gamma^*\backslash \mathcal D] ,
\end{equation*} 
where $\Gamma^*$ is  the automorphism group  of the  ({\it parahoric}) lattice $L^*$ corresponding to the $*$-moduli problem.  
The lattice $L^*$  is  uniquely determined up to isomorphism  by the condition 
that there is a chain of inclusions of $\OK$-lattices 
$L^*\subset (L^*)^\vee\subset (\sqrt{d})^{-1}L^*$, with quotient $(L^*)^\vee /L^*$ of dimension $n-1$ if $n$ is odd and $n-2$ if $n$ is even, when localized 
at any prime ideal $\mathfrak p$ dividing $d$.  Here, for an $\OK$-lattice $M$ in $V$, 
we write 
$$M^\vee = \{ \ x\in V\ \mid \ h(x,L) \subset \OK\ \}$$
for the dual lattice.

\section{ Special cycles (KM-cycles)}\label{sc}
 We continue to assume that the class number of $\kay$ is one, and recall from \cite{KR2} the definition of special cycles over 
 $\C$. Let $(E, \iota_0)$ be an elliptic curve with $CM$ by $\OK$ over $\C$, which we fix in what follows. 
 Note that, due to our class number hypothesis, $(E, \iota_0)$
is unique up to isomorphism. We denote its canonical  principal polarization by $\lambda_0$. 
For any $\C$-scheme $S$, and $(A, \iota, \lambda)\in\M (\kay; n-1, 1) (S)$, let
\begin{equation*}
V' (A, E) = \Hom_{\OK} (E_S, A)\ ,
\end{equation*}
where $E_S = E\times_\C S$ is the constant elliptic scheme over $S$ defined by $E$. 
Then $V' (A, E)$ is a projective $\OK$-module of finite rank  with a positive definite $\OK$-valued hermitian form given by
\begin{equation*}
h' (x, y) = \lambda_0^{-1}\circ y^\vee\circ\lambda\circ x\in\End_{\OK} (E_S) = \OK\ .
\end{equation*}
For a positive integer $t$, we define the DM-stack\footnote{This notation 
differs from that in \cite{KR2}, in that here the special cycles are defined over $\C$, and are considered as lying over $\M(\smallkay; n-1, 1)_\C$.}
$\ZZ(t)$ by 
\begin{equation*}
\ZZ(t)(S) = \{(A, \iota, \lambda; x)\mid (A, \iota, \lambda) \in \M (\kay; n-1, 1)(S), \  x\in V' (A, E), \ \ h' (x, x) = t\}\ .
\end{equation*}
Then $\ZZ(t)$ maps by a finite unramified morphism to $\M (\kay; n-1, 1)_\C$,  and its image is a divisor in the sense that, locally 
for the \'etale topology,  it is defined by a non-zero equation. 

The cycles $\ZZ(t)$ also admit a complex uniformization. More precisely, under the assumption of the triviality of the class group of $\kay$, we have
\begin{equation*}
\ZZ(t)(\C) \simeq \Big[\Gamma\backslash\Big(\coprod_{\substack{x\in L\\ h(x, x)=t}}\mathcal D_x\Big)\Big] ,
\end{equation*} 
where $\mathcal D_x$ is the set of lines in $\mathcal D$ which are perpendicular to $x$. 

Again, there is a $*$-variant of these definitions and a corresponding DM-stack $\ZZ(t)^*$ above $\M(\kay; n-1, 1)^*$.

\section{ Cubic surfaces} \label{cubicsurfaces}
In this paper we  consider  four occult  period mappings. We start with the case of cubic surfaces, following Allcock, Carlson, Toledo \cite{ACT1}, comp.\ also \cite{B}. 
As explained in the introduction, in these sources, the results are formulated in terms of arithmetic ball quotients;
here we use the complex uniformization of the previous two sections to
express these results in terms of moduli spaces of Picard type.

Let $S\subset\mathbb P^3$ be a smooth cubic surface. Let $V$ be a cyclic covering of degree $3$ of $\mathbb P^3$, 
ramified along $S$. Explicitly, if $S$ is defined by the homogeneous equation of degree $3$ in $4$
variables
\begin{equation*}
F (X_0, \ldots , X_3) = 0\ ,
\end{equation*}
then $V$ is defined by the homogeneous equation of degree $3$
in $5$ variables,
\begin{equation*}
X_4^3 - F(X_0, \ldots , X_3) = 0\ .
\end{equation*}

Let $\kay = \Q (\o)$, $\o = e^{2\pi i /3}$. Then the obvious $\mu_3$-action on $V$ determines an action of $\OK=\Z[\o]$ on $H^3 (V, \Z)$. 
For  the (alternating) cup product pairing $\langle \ , \ \rangle$,
\begin{equation*}
\langle \o x, \o y \rangle = \langle  x,  y \rangle  ,
\end{equation*}
which implies that
\begin{equation*}
\langle \aa x,  y \rangle = \langle  x,  {\aa^\sigma} y \rangle ,\quad \forall\,\aa \in \OK . 
\end{equation*}
Hence there is a unique $\OK$-valued hermitian form $h$ on $H^3 (V, \Z)$ such that 
\begin{equation}\label{altherm}
\langle  x,  y \rangle = {\rm tr} \big(\frac{1}{\sqrt \Delta} h( x,  y)\big) ,
\end{equation}
where the discriminant $\Delta$   of $\kay$ is equal to $-3$ in the case at hand. Explicitly,
\begin{equation}
h(x, y)=\frac{1}{2}\big(\langle \sqrt{\Delta} x,  y \rangle +\langle  x,  y \rangle \sqrt{\Delta}\big) .  
\end{equation} 

Furthermore, an $\OK$-lattice is self-dual wrt $\langle\ ,\ \rangle$ if and only if it is self-dual wrt $h(\ ,\ )$. 
\smallskip

{\bf Fact:} {\it $H^3 (V, \Z)$ is a self-dual hermitian $\OK$-module of signature $(4, 1)$.}
\smallskip

As noted above, such a lattice is unique up to isomorphism.

Let
\begin{equation*}
A = A(V) = H^3 (V, \Z)\backslash H^3 (V, \C) / H^{2, 1} (V)
\end{equation*}
be the intermediate Jacobian of $V$. Then $A$ is an abelian variety of dimension $5$ which is principally
polarized by the intersection form. Since the association $V\mapsto (A(V), \lambda)$ is functorial, we obtain
an action $\iota$ of $\OK$ on $A(V)$.
\begin{theo}\label{cubicsurf}
(i) The object $(A, \iota, \lambda )$ lies in $\M (\kay; 4,1) (\C)$.
\smallskip

\noindent (ii)This construction is functorial and compatible with families,    
and defines a morphism of
DM-stacks,
\begin{equation*}
\varphi: {\it Cubics}_{2, \C}^\circ\to\M (\kay; 4,1)_\C\ .
\end{equation*}
{Here ${\it Cubics}_{2, \C}^\circ$ denotes the stack parametrizing smooth cubic surfaces up to projective 
equivalence,
\begin{equation*}
\text{\it Cubics}_{2, \C}^\circ = [ {\rm PGL}_4\backslash \mathbb P\Sym^3 (\C^4)^\circ ]
\end{equation*}
{\rm [stack quotient in the orbifold sense].}}
\smallskip

\noindent (iii) The induced morphism on coarse moduli spaces $|\varphi|: |{\it Cubics}_{2, \C}^\circ|\to|\M (\kay; 4,1)_\C|$ is an open embedding. Its image is the complement of the image of the 
KM-cycle $\ZZ(1)$ in $|\M (\kay; 4,1)_\C|$.
\end{theo}
\begin{proof}
We only comment  on  the assertions in (ii) and (iii). In (ii), the compatibility with families is always true of Griffiths'  
intermediate jacobians (which however are abelian varieties only when the Hodge structure is of type 
$(m+1, m)+(m, m+1)$). This constructs $\varphi$ as a complex-analytic morphism. 
The algebraicity of $\varphi$ then follows from Borel's theorem that any analytic family of abelian varieties over a $\C$-scheme is automatically algebraic \cite{Bo}. 
The fact that the image is contained in the complement of $\ZZ(1)$ is true because, by the Clemens-Griffiths 
theory, intermediate Jacobians of cubic threefolds  are simple as polarized abelian varieties, whereas,  over $\ZZ(1)$ the 
polarized abelian varieties split off an elliptic curve. However, the fact that $\ZZ(1)$ makes up the whole complement 
is surprising and results from the fact that the morphism $\varphi$ extends to an isomorphism from a 
partial compactification $|\text{\it Cubics}_{2, \C}^{\rm s}|$ of $|\text{\it Cubics}_{2, \C}^\circ|$ (obtained by 
adding {\it stable} cubics) to  $|\M (\kay; 4,1)_\C|$, such that the complement of 
$|\text{\it Cubics}_{2, \C}^\circ|$ in $|\text{\it Cubics}_{2, \C}^{\rm s}|$  is an irreducible divisor, cf.\ \cite{B}, Prop.\ 6.7, Prop.\ 8.2. 
\end{proof}
\begin{rem} {\rm Let us comment on the stacks aspect of Theorem \ref{cubicsurf}. Any automorphism of $S$ is induced by an automorphism of $\mathbb P^3$, which in turn induces an automorphism of $V$. We therefore obtain a homomorphism
$ {\rm Aut}(S)\to {\rm Aut}(A(V), \iota, \lambda).$
The statement of 
\cite{ACT1}, Thm.~2.20. implies that this homomorphism induces an isomorphism
 \begin{equation}\label{stackiso}
  {\rm Aut}(S)\overset{\sim}{\longrightarrow} {\rm Aut}(A(V), \iota, \lambda)/\OK^\times\, ,
\end{equation} 
where the units $\OK^\times\simeq\mu_6$ act via $\iota$  on $A(V)$. Indeed, in loc. cit. it is asserted that $\varphi$ is an open immersion of orbifolds ${\it Cubics}_{2, \C}^\circ\to [P\Gamma\backslash \mathcal D]$, where $P\Gamma=\Gamma/\OK^\times$; however, we were not able to follow the argument. Note that the orbifold  $[P\Gamma\backslash \mathcal D]$ is different from $[\Gamma\backslash \mathcal D]$, which occurs in \S 3. 
}
\end{rem}
\section{ Cubic threefolds}\label{cubicthreefolds} 
Our next example concerns cubic threefolds, following Allcock, Carlson, Toledo  \cite{ACT2} and Looijenga, Swierstra \cite{LS1}.

Let $T\subset\mathbb P^4$ be a cubic threefold. Let $V$ be the cyclic covering of degree $3$ of $\mathbb P^4$, 
ramified in $T$. Then $V$ is a cubic hypersurface in $\mathbb P^5$ and we define the primitive cohomology as 
\begin{equation}
L=H_0^4 (V, \Z) = \{x\in H^4 (V, \Z)\mid ( x, \rho) = 0\}\ ,
\end{equation}
where $\rho$ is the square of the hyperplane section class. Note that  ${\rm rk}_{\Z} L = 22$. 
Again, let $\kay =\Q (\o)$, with $\o= e^{2\pi i/3}$, so that  $L$
becomes an $\OK$-module. Now the cup-product $(\ , \  )$ on $H^4 (V, \Z)$  is a perfect {\it symmetric} 
pairing satisfying $(a x, y)= (x, {a^\sigma} y)$ for $a \in\OK$. It induces on
$L$ a symmetric bilinear form $(\  ,\  )$ of discriminant $3$. We wish to define an {\it alternating} pairing 
$\langle \ , \ \rangle$ on $L$ satisfying $\langle \aa x, y\rangle=\langle x, {\aa^\sigma} y\rangle$ for $\aa\in \OK$. 
We do this by giving the associated $\OK$-valued hermitian pairing $h(\ , \ )$, 
in the sense of (\ref{altherm}) defined by 
\begin{equation}\label{hermit1}
h(x, y) = \frac{3}{2}  \big(( x, y) +  ( x, \sqrt{\Delta} y) \frac{1}{\sqrt{\Delta}}\big). 
\end{equation}
Here the factor $3/2$ is used instead of $1/2$  to have better integrality properties. Set $\pi=\sqrt{\Delta}$.

{\bf Fact:} {\it For the pairing \eqref{hermit1},  $L^{\vee}$ contains $\pi^{-1}L$ with 
$L^\vee/\pi^{-1}L\simeq \Z/3\Z$.} \hfb 
For this result, see \cite{ACT2}, Theorem\ 2.6 and its proof, as well as \cite{LS1}, the passage below (2.1).

Now consider the eigenspace decomposition of $H_0^4 (V, \C)$ under $\kay\otimes\C = \C\oplus\C$.

{\bf Fact:} {\it The Hodge structure of $H_0^4 (V, \R)$ is of type
\begin{equation*}
H_0^4 (V, \C) = H^{3, 1}\oplus H_0^{2, 2}\oplus H^{1, 3}\ ,
\end{equation*}
with $\dim H^{3, 1} = \dim H^{1, 3} = 1$. Furthermore, the only nontrivial eigenspaces of the generator $\o$ of $\mu_3$ are
\begin{equation*}
\begin{aligned}
H_0^4 (V, \C)_\o & = & H^{3, 1}\oplus (H_0^{2, 2})_\o\ , & \text{ with}\  \dim (H_0^{2, 2})_\o = 10\\
H_0^4 (V, \C)_{\overline{\o}} & = & (H_0^{2, 2})_{\overline{\o}}\oplus H^{1, 3}\ , & \text{ with}\ \dim (H_0^{2, 2})_{\overline{\o}} = 10\ ,
\end{aligned}
\end{equation*}}
see \cite{ACT2}, \S 2, resp.\ \cite{LS1} \S 4.

\smallskip

Now set $\Lambda=\pi L^\vee$. Then we have the chain of inclusions of $\OK$-lattices
\begin{equation*}
\Lambda\subset \Lambda^\vee\subset \pi^{-1}\Lambda\  ,
\end{equation*}
where the quotient $\Lambda^\vee/\Lambda$ is isomorphic to $(\Z/3\Z)^{10}$, and where $\pi^{-1}\Lambda/\Lambda^\vee$ is isomorphic to $\Z/3\Z$. 
Let
\begin{equation*}
A = \Lambda\backslash H_0^4 (V, \C) / H^-\ ,
\end{equation*}
where
\begin{equation*}
H^- = H^{3, 1}\oplus (H_0^{2, 2})_{\overline{\o}}\ .
\end{equation*}
Note that the map $\Lambda\to H_0^4 (V, \C) / H^-$ is an $\OK$-linear injection, hence $A$ is 
a complex torus. In fact, the hermitian form $h$ and its associated alternating form $\langle \ , \ \rangle$ define a polarization $\lambda$ on $A$. Hence $A$ is an abelian variety of dimension $11$, with an action of $\OK$ and a polarization of degree
$3^{10}$. In fact, we obtain in this way an object $(A, \iota, \lambda )$ of $\M (\kay; 10, 1)^\ast (\C)$ (see section \ref{subsectionpic} for the definition of the $*$-variants of our moduli stacks).
\begin{theo}\label{thmcubic3}

\noindent (i) The construction which associates to a smooth cubic $T$ in $\mathbb P^4$ the object $(A, \iota, \lambda )$
of $\M (\kay; 10, 1)^\ast(\C)$ is functorial and compatible with families,  and defines a morphism of DM-stacks,
\begin{equation*}
\varphi : \text{Cubics}_{3, \C}^\circ\to\M (\kay; 10,1)_\C^\ast\ .
\end{equation*}

\noindent (ii) The induced morphism on coarse moduli spaces $|\varphi| : |\text{Cubics}_{3, \C}^\circ|\to|\M (\kay; 10,1)_\C^\ast| $ is an open embedding. Its image is the complement of the  image of the KM-cycle $\ZZ (3)^*$ in $|\M (\kay; 10,1)_\C^\ast|$.
\end{theo}
\begin{proof}
 The compatibility with families is due to the fact 
that the eigenspaces for the $\mu_3$-action and the Hodge filtration both vary in a 
holomorphic way. The point (ii)  is \cite{ACT2}, Thm.\ 1.1, resp.\ \cite{LS1}, Thm.\ 3.1.  
\end{proof}
\begin{rem}{\rm  The stack aspect is not treated in these sources. However, it seems reasonable to conjecture that the analogue of \eqref{stackiso} is also true in this case, i.e., that there is an isomorphism 
 \begin{equation}\label{stackiso2}
  {\rm Aut}(T)\overset{\sim}{\longrightarrow} {\rm Aut}(A, \iota, \lambda)/\OK^\times\, ,
\end{equation} 
where $(A, \iota, \lambda)$ is the object of $\M (\kay; 10,1)_\C^\ast$ attached to $T$. }\end{rem}
\begin{rem}
{\rm The construction of the rational Hodge structure $H^1 (A, \Q)$ from $H_0^4 (V, \Q)$ is a very special case
of a general construction due to van Geemen \cite{G}. More precisely, it arises (up to Tate twist) as the  {\it inverse
half-twist} in the sense of loc.\ cit.\  of the Hodge structure $H_0^4 (V, \Q)$ with complex 
multiplication by $\kay$. The {\it half twist} construction attaches to a rational Hodge structure 
$V$ of weight $w$ with complex multiplication  by a CM-field $\kay$ a rational Hodge 
structure of weight $w+1$. More precisely, if $\Sigma$ is a fixed half system of complex 
embeddings of $\kay$, then van Geemen defines a new Hodge structure on $V$ by setting
\begin{equation*}
V_{\rm new}^{r, s}=V_{\Sigma}^{r-1, s}\oplus V_{\overline{\Sigma}}^{r, s-1} , 
\end{equation*}  
where $V_{\Sigma}$, resp.\ $V_{\overline{\Sigma}}$ denotes the sum of the 
eigenspaces for the $\kay$-action corresponding to the complex embeddings in ${\Sigma}$, resp. in ${\overline{\Sigma}}$. }
\end{rem}

\section{Curves of genus three} \label{curvesgenus3}

Our third example concerns the moduli space of curves of genus $3$ following Kondo \cite{K1}. 

Let $C$ be a non-hyperelliptic smooth projective curve of genus 3. The canonical system embeds $C$ as a
quartic curve in $\mathbb P^2$. Let $X(C)$ be the $\mu_4$-covering of $\mathbb P^2$ ramified in $C$. Then the quartic $X(C)\subset \mathbb P^3$ is a
$K3$-surface with an automorphism $\tau$ of order $4$ and hence an action of $\mu_4$. Let
\begin{equation*}
L= \{x\in H^2 (X (C), \Z)\mid\tau^2 (x) = -x\}\ .
\end{equation*}
Let $\kay = \Q (i)$ be the Gaussian field.

\smallskip

{\bf Fact:} {\it $L$ is a free $\Z$-module of rank $14$. The restriction $(\ , \ )$ of the symmetric
cup product pairing to $L$ has discriminant $2^8$; more precisely, for the dual lattice $L^*$ for the symmetric pairing, 
\begin{equation*}
L^* / L\cong (\Z/2)^8\ ,
\end{equation*}}
see \cite{K1}, top of p.\ 222.

\smallskip

Now consider the eigenspace decomposition of $L_{ \C}=L\otimes\C$ under $\kay\otimes\C = \C\oplus\C$, where $i\otimes 1$ acts via $\tau$.

\smallskip

{\bf Fact:} {\it The induced Hodge structure on $L_{\C}$ is of type
\begin{equation*}
L_{\C} = L^{2,0}\oplus L^{1,1}\oplus L^{0,2}\ ,
\end{equation*}
with $\dim L^{2,0} = \dim L^{0,2} = 1$. Furthermore the only nontrivial eigenspaces of $\tau$ are
\begin{equation*}
\begin{aligned}
(L_{\C})_i & = & L^{2,0}\oplus (L^{1,1})_i\ , & \text{ with}\ \dim (L^{1,1})_i = 6\\
(L_{\C})_{-i} & = & (L^{1,1})_{-i}\oplus L^{0,2}\ , & \text{ with}\ \dim (L^{1,1})_{-i} = 6\ .
\end{aligned}
\end{equation*}}

\smallskip

We define an $\OK$-valued  hermitian pairing $h$ on $L_\Q$ by setting 
\begin{equation}
h(x, y)= (x, y) +(x, \tau y)\ i\  .
\end{equation}
Then it is easy to see that the dual lattice $L^\vee$ of $L$ for the hermitian form $h$ is the same as the dual lattice $L^*$ for the symmetric form. 

Now set $\Lambda=\pi L^\vee$, where $\pi=1+ i$. Then we obtain a chain of inclusions of $\OK$-lattices
\begin{equation*}
\Lambda\subset \Lambda^\vee\subset \pi^{-1}\Lambda\  ,
\end{equation*}
where the quotient $\Lambda^\vee/\Lambda$ is isomorphic to $(\Z/2\Z)^{6}$, and where $\pi^{-1}\Lambda/\Lambda^\vee$ is isomorphic to $\Z/2\Z$.

Let
\begin{equation*}
A = \Lambda\backslash L_{\C} / L^-\ ,
\end{equation*}
where
\begin{equation*}
L^-= L^{2,0}\oplus (L^{1,1})_{-i}\ .
\end{equation*}
Note that the map $\Lambda\to L_{\C} / L^-$ is a $\OK$-linear injection, hence
$A$ is a complex torus. In fact, the hermitian form $h$ and its associated 
alternating form $\langle \ , \ \rangle$ define a polarization $\lambda$ 
on $A$. Hence $A$ is an abelian variety of dimension $7$, with an action of $\OK$ and a polarization of degree
$2^{6}$. In fact, we obtain in this way an object $(A, \iota, \lambda )$ 
of $\M (\kay; 6, 1)^\ast (\C)$. 
Now \cite{K1}, Thm.\ 2.5  implies the following theorem. 

\begin{theo}

\noindent (i) The construction which asssociates to a non-hyperelliptic curve of genus $3$ the object $(A, \iota, \lambda )$
of $\M (\kay; 6,1)^\ast(\C)$ is functorial and compatible with families,  and defines a morphism of DM-stacks,
\begin{equation*}
\varphi : \N_{3, \C}^\circ\to\M (k; 6,1)^\ast_\C\ .
\end{equation*}
{ Here $\N_{3,\C}^\circ$ denotes the stack of smooth non-hyperelliptic curves of genus $3$, i.e., of smooth non-hyperelliptic quartics in $\mathbb P^2$ up to projective equivalence.}
\smallskip

\noindent (ii) The induced morphism on coarse moduli schemes $|\varphi| : |\N_{3, \C}^\circ|\to|\M (k; 6,1)^\ast_\C| $ is an open embedding. Its image is the complement of the image of the KM-cycle $\ZZ (2)^*$ in $|\M (k; 6,1)^\ast_\C|$.\qed
\end{theo}
\begin{rem}{\rm  Again, the stack aspect is not treated in \cite{K1}. It seems reasonable to conjecture that the analogue of \eqref{stackiso} is also true in this case, i.e., that there is an isomorphism 
 \begin{equation}\label{stackiso2}
  {\rm Aut}(C)\overset{\sim}{\longrightarrow} {\rm Aut}(A, \iota, \lambda)/\OK^\times\, ,
\end{equation} 
where $(A, \iota, \lambda)$ is the object of $\M (\kay; 6,1)_\C^\ast$ attached to $C$, and where $\OK^\times=\mu_4$. }\end{rem}

\section{Curves of genus four} \label{curvesgenus4}
Our final example concerns the moduli space of curves of genus four and is also due to Kondo  \cite{K2}. 

Let $C$ be a non-hyperelliptic curve of genus $4$. The canonical system embeds $C$ into
$\mathbb P^3$. More precisely, $C$ is the intersection of a smooth cubic surface $S$ and a
quartic $Q$ which is either smooth or a quadratic cone. Furthermore, $Q$ is uniquely determined by $C$. Let $X$ be a cyclic cover of
degree $3$ over $Q$ branched along $C$ (in case $Q$ is singular, we take the minimal
resolution of the singularities, cf.\ loc.cit.). Then $X$ is a $K3$-surface with an action of $\mu_3$. Let
\begin{equation*}
L = (H^2 (X, \Z)^{\mu_3})^\perp
\end{equation*}
be the orthogonal complement of the invariants of this action in $H^2 (X, \Z)$, equipped with the symmetric form $(\ , \ )$ obtained by restriction. 

\smallskip

{\bf Fact:} {\it $L$ is a free $\Z$-module of rank $20$, with dual $L^*$ for the symmetric form satisfying 
\begin{equation*}
L^* / L \simeq (\Z /3\Z)^2\ ,
\end{equation*}}
cf. \cite{K2}, top of p.\ 386. 

\smallskip

For $\kay=\Q(\o)$, $\o= e^{2\pi i/3}$, we again define an alternating form $\langle \ , \ \rangle$ through its associated 
$\OK$-valued
hermitian form $h$. Using the action of $\OK$ on $L$,  we set
\begin{equation}\label{hermit2}
h(x, y) = \frac{3}{2}  \big(( x, y) +  ( x, \sqrt{\Delta} y) \frac{1}{\sqrt{\Delta}}\big)\ .
\end{equation}
Set $\pi=\sqrt{\Delta}$. 

\smallskip

{\bf Fact:} {\it For the hermitian pairing \eqref{hermit2},  $L^{\vee}$ is  an over-lattice of $\pi^{-1}L$ with $L^\vee/\pi^{-1}L\simeq (\Z/3\Z)^2$.}

\smallskip

Now consider the eigenspace decomposition of $L\otimes\C$ under $\kay\otimes\C = \C\oplus\C$.

\smallskip

{\bf Fact:} {\it The induced Hodge structure on $L_\C$ is of type
\begin{equation*}
L_\C = L^{2,0}\oplus L^{1,1}\oplus L^{0,2}\ ,
\end{equation*}
with $\dim L^{2,0} = \dim L^{0,2} = 1$. Furthermore the only nontrivial eigenspaces of $\mu_3$ are 
\begin{equation*}
\begin{aligned}
(L_\C)_\omega = L^{2,0}\oplus (L^{1,1})_\o\ ,\ \text{with}\ \dim (L^{1,1})_\o &= 9\\
(L_\C)_{\overline{\omega}} = (L^{1,1})_{\overline{\o}}\oplus L^{0,2}\ ,\ \text{with}\ \dim (L^{1,1})_{\overline{\o}} &= 9\ .
\end{aligned}
\end{equation*}}

\smallskip

Now set $\Lambda=\pi L^\vee$. Then we have the chain of inclusions of $\OK$-lattices
\begin{equation*}
\Lambda\subset \Lambda^\vee\subset \pi^{-1}\Lambda\  ,
\end{equation*}
where the quotient $\Lambda^\vee/\Lambda$ is isomorphic to $(\Z/3\Z)^{8}$, and where $\pi^{-1}\Lambda/\Lambda^\vee$ is isomorphic to $(\Z/3\Z)^2$.

Let
\begin{equation*}
A = \Lambda\backslash L_\C / L^-\ ,
\end{equation*}
where 
\begin{equation*}
L^- = L^{2,0}\oplus (L^{1,1})_{\overline{\o}}\ .
\end{equation*}

 Then the map
$\Lambda\to L_\C / L^-$ is a $\OK$-linear injection, hence $A$ is a complex torus. 
In fact, the hermitian form $h$ and its associated alternating form $\langle \ , \ \rangle$ 
define a polarization $\lambda$ on $A$. Hence $A$ is an abelian variety of dimension $10$, with an action of $\OK$ and a polarization of degree
$3^{8}$. In fact, we obtain in this way an object $(A, \iota, \lambda )$ of $\M (\kay; 9, 1)^\ast (\C)$, 
\begin{theo}

\noindent (i) The construction which associates to a non-hyperelliptic curve of genus $4$ the object
$(A,\iota, \lambda )$ of $\M (\kay; 9,1)^\ast(\C)$ is functorial and compatible with families,  and defines a morphism of
DM-stacks,
\begin{equation*}
\varphi :\N^\circ_{4, \C}\to\M (\kay; 9,1)^\ast_\C\ .
\end{equation*}
{\rm Here $\N^\circ_{4, \C}$ denotes the stack of smooth non-hyperelliptic curves of genus $4$.}

\smallskip

\noindent (ii)  The induced morphism on coarse moduli schemes $|\varphi| :|\N^\circ_{4, \C}|\to|\M (\kay; 9,1)^\ast_\C| $ is an open embedding. Its image is the complement of the image of the KM-cycle
$\ZZ (2)^*$ in $|\M (\kay; 9,1)^\ast_\C|$.\qed
\end{theo}
\begin{rem}{\rm  Again, the stack aspect is not treated  in \cite{K2}. It seems reasonable to conjecture that the analogue of \eqref{stackiso} is also true in this case, i.e., that there is an isomorphism 
 \begin{equation}\label{stackiso2}
  {\rm Aut}(C)\overset{\sim}{\longrightarrow} {\rm Aut}(A, \iota, \lambda)/\OK^\times\, ,
\end{equation} 
where $(A, \iota, \lambda)$ is the object of $\M (\kay; 9,1)_\C^\ast$ attached to $C$, and where $\OK^\times=\mu_6$. }\end{rem}

\section{Descent} \label{descent}
In all four cases discussed above, we obtain  morphisms over $\C$ between 
DM-stacks defined over $\kay$. These 
morphisms are constructed using transcendental methods. In this section 
we will show that  these morphisms are in fact defined over $\kay$.  The argument is 
modelled on Deligne's solution of the analogous problem for complete intersections of 
Hodge level one \cite{De1}, where he shows that the corresponding family of intermediate 
jacobians is an abelian scheme over the moduli scheme over $\Q$ of complete intersections of given multi-degree. 

In our discussion below, to simplify notations,  we will deal with  the case of cubic threefolds, as explained in 
section \ref{cubicthreefolds}; the other cases are completely analogous. Below we will shorten the 
notation $Cubics^\circ_3$ to $\mathcal C$, and consider this as a DM-stack over $\Spec \kay$. 
Let $v:V\to \mathcal C$ be the universal family of cubic threefolds, and let $a: A\to \mathcal C_\C$ 
be the polarized family of abelian varieties constructed from $V$ in section \ref{cubicthreefolds}. 
Hence $A$ is the pullback  of the universal abelian scheme over $\mathcal M(\kay; 10, 1)^*_\C$ under the 
morphism  $\varphi: \mathcal C_\C\to \mathcal M(\kay; 10, 1)^*_\C$.  
\begin{lem}\label{elladic}
 Let $b: B\to \mathcal C_\C$ be a polarized abelian scheme with $\OK$-action, which is the pullback 
 under a morphism $\psi: \mathcal C_\C\to \M(\kay; 10, 1)_\C^*$ of the universal abelian scheme, 
 and such that there exists $\ell$ and an $\OK$-linear isomorphism of lisse $\ell$-adic sheaves on $\mathcal C_\mathbb C$, 
$$
\alpha_{\ell}: R^1a_*\Z_{\ell}\simeq R^1b_*\Z_{\ell}
$$
compatible with the Riemann forms on source and target. Then there exists a unique 
isomorphism $\alpha: A\to B$ that induces $\alpha_{\ell}$. This isomorphism is compatible with polarizations. 
\end{lem}

To prove this,  we are going to use the following lemma. In it, we denote by $\Lambda$ the 
hermitian $\OK$-module $H^1(A_s, \Z)$, for $s\in\mathcal C_\C$  a fixed base point. 
Recall from section \ref{cubicthreefolds} that there is a chain 
of inclusions $\Lambda\subset \Lambda^\vee\subset \pi^{-1}\Lambda$, where $\pi=\sqrt {-3}$ is a generator of the unique prime ideal 
of $\OK$ dividing $3$. 
\begin{lem}\label{monodr}
Let  $s\in\mathcal C_\C$ be the chosen base point.

\noindent (i) The monodromy representation  $\rho_A:\pi_1(\mathcal C_\C, s)\to {\rm GL}_{\smallkay}\big(\Lambda\otimes_{\OK}\kay\big)$ is absolutely irreducible.

\noindent (ii) For every prime ideal $\mathfrak p$ prime to  $3$, the monodromy 
representation  $\pi_1(\mathcal C_\C, s)\to {\rm GL}_{\kappa(\mathfrak p)}\big(
\Lambda/\mathfrak p \Lambda\big)$ is absolutely irreducible.  

\noindent (iii) For the unique prime ideal $\mathfrak p=(\pi)$ lying over $3$, the 
monodromy representation $\pi_1(\mathcal C_\C, s)\to {\rm GL}_{\kappa(\mathfrak p)}\big(
\Lambda/\mathfrak p \Lambda\big)$  is not absolutely irreducible, but there is  a 
unique non-trivial stable subspace, namely, the $10$-dimensional image of $\pi \Lambda^\vee$ in $\Lambda/\pi \Lambda$.    
\end{lem}
\begin{proof} The monodromy representations in question are induced by the composition of homomorphisms
\begin{equation}\label{monodromy}
\pi_1(\mathcal C_\C, s)\longrightarrow \pi_1(\mathcal M(\kay; 10, 1)^*_\C, \ph(s))\longrightarrow \GL_{\OK}\big(H^1(A_s, \Z)\big) .
\end{equation}
Here by Theorem \ref{thmcubic3}, and using complex uniformization (cf. section \ref{cu}),  the first 
homomorphism is induced by the inclusion of connected spaces 
$$\iota: \mathcal D\setminus \Big(\bigcup_{\substack{x\in L\\ h(x, x)=3}}\mathcal D_x\Big) \hookrightarrow \mathcal D\ , 
$$ followed by quotienting out by the free action of $\Gamma^*$. Since $\mathcal D$ is simply-connected,  it follows 
that $\pi_1(\mathcal M(\kay; 10, 1)^*_\C, \ph(s))=\Gamma^*$ and that the first homomorphism 
in (\ref{monodromy}) is surjective. Now, $\Gamma^*$ can 
be identified with the  group  of unitary automorphisms of the {\it parahoric}  lattice $\Lambda$, 
and it is elementary that the representations of $\Gamma^*$ on $\Lambda\otimes_{\OK}\kay$ 
and on $\Lambda/\mathfrak p \Lambda$ for $\mathfrak p$ prime to $ 3$ 
are absolutely irreducible (the latter since $\Lambda^\vee\otimes {\Z}_\ell=\Lambda\otimes {\Z}_\ell$ for $\ell\neq 3$). 
The statement (iii) is proved in the same way. 
\end{proof}
\begin{proof}(of Lemma \ref{elladic}) 
Let us compare the monodromy representations, 
\begin{equation}
\begin{aligned}
\rho_A:& \pi_1(\mathcal C_\C, s)\to \GL_{\OK}\big(H^1(A_s, \Z)\big)\\
\rho_B:& \pi_1(\mathcal C_\C, s)\to \GL_{\OK}\big(H^1(B_s, \Z)\big) .
\end{aligned}
\end{equation}
By hypothesis, these representations are isomorphic after tensoring with $\Z_{\ell}$. 
Hence, they are also isomorphic after tensoring with  $\kay$. Hence there 
exists a $\pi_1(\mathcal C_\C, s)$-equivariant $\kay$-linear isomorphism 
\begin{equation*}
\beta: H^1(A_s, \Q)\simeq H^1(B_s, \Q)\ .
\end{equation*}
By the irreducibility of the representation of $\pi_1(\mathcal C_\C, s)$ in $H^1(A_s, \Q)$, $\beta$ is unique 
up to a scalar in $\kay^\times$. Let us compare the $\OK$-lattices 
$\beta^{-1}\big(H^1(B_s, \Z)\big)$ and $H^1(A_s, \Z)$. Since we are assuming that $\OK$ is a PID, after replacing $\beta$ by a multiple $\beta_{\OO}=c\beta$, 
we may assume that $L_B= \beta_{\OO}^{-1}\big(H^1(B_s, \Z)\big)$ 
is a primitive $\OK$-sublattice in $\Lambda=H^1(A_s, \Z)$. Let $\mathfrak p$ be a prime ideal in $\OK$, 
and let us consider the image of $L_B$ in $\Lambda/\mathfrak p \Lambda$. Since $L_B$ is 
primitive in $\Lambda$, this image is non-zero. If $\mathfrak p$ is prime to $3$, the irreducibility 
statement in (ii) of Lemma \ref{monodr} implies that this image is everything, and hence
$L_B\otimes O_{\smallkay, \mathfrak p}=\Lambda\otimes O_{\smallkay, \mathfrak p}$ in this case. 

To handle the prime ideal $\mathfrak p=(\pi)$ over $3$, we use the polarizations. By the irreducibility 
statement in (i) of Lemma \ref{monodr}, the polarization forms on $H^1(A_s, \Q)$ and on $H^1(B_s, \Q)$ 
differ by a scalar in $\Q^\times$ under the isomorphism $\beta_{\OO}$. Now, by hypothesis on $B$, with respect to the polarization form 
on $H^1(B_s, \Q)$, we have a chain of inclusions $L_B\subset L_B^\vee\subset \pi^{-1}L_B$ 
with respective quotients of dimension $10$ and $1$ over $\F_p$, just as for $\Lambda$. Since the 
two polarization forms differ by a scalar, this excludes the possibility that the image 
of $L_B$ in $\Lambda/\pi \Lambda$ be non-trivial. It follows that $L_B=\Lambda$. 

Furthermore, the isomorphism $\beta_{\OO}$ is unique up to a unit in $\OK^\times$, and it is 
an isometry with respect to both polarization forms. Now, by \cite{De2}, 4.4.11 and 4.4.12, $\beta_\OO$ is induced 
by an isomorphism of polarized abelian schemes. Finally, $\beta_\OO\otimes_\Z\Z_\ell=\alpha_\ell$ up to a unit, 
since these homomorphisms differ by a scalar and both preserve the Riemann forms.\hfb 
The uniqueness of $\alpha$ follows from Serre's Lemma. 
\end{proof}
Now Lemma \ref{elladic} implies  that over any  field extension $k'$ of $\kay$ inside $\C$, there exists  
at most one  polarized abelian variety $b: B\to \mathcal C_{k'}$ obtained by pull-back from the universal 
abelian variety over $\M(\kay; 10, 1)^*$,  equipped with an $\OK$-linear isomorphism of lisse $\ell$-adic sheaves over $\mathcal C_\C$
$$ R^1a_{*}\Z_{\ell}\simeq R^1b_{\C *}\Z_{\ell}\ ,$$
preserving the Riemann forms. By the argument in \cite{De1}, 2.2 this implies that, in fact, $B$ 
exists (since it does for $k'=\C$). Hence the morphism $\varphi$ is defined over $\kay$. Put otherwise, 
for any $\kay$-automorphism $\tau$ of $\C$, the conjugate embedding $\varphi^\tau$, which corresponds 
to the conjugate $(A, \iota, \l)^\tau$, is equal to $\varphi$;  hence $\varphi$ is defined over $\kay$.   

\begin{conj}
In all four cases above, the  morphisms $\varphi$ can  be extended over $\OK[\Delta^{-1}]$.
\end{conj}

Since we circulated a first version of our paper, this has been proved by J.~Achter \cite{Ach} in the case of cubic surfaces. 

\section{Concluding remarks} 

We end this paper with a few remarks. 
\begin{rem}
{\rm In all four cases, the complement of ${\rm Im}(|\varphi|)$ is identified with a certain KM-divisor. In fact,
 for other KM-divisors, the intersection with ${\rm Im}(|\varphi|)$ sometimes has a geometric interpretation. For example,
 in the case of cubic surfaces, the intersection of ${\rm Im}(|\varphi|)$ with the image of the KM-divisor $\ZZ(2)$ in $|\M(\kay; 4, 1)_{\C}|$
 can be identified with the locus of cubic surfaces admitting Eckardt points, cf. \cite{DGK}, Thm.~8.10. Similarly, in the case of curves of genus $3$,
 the intersection of ${\rm Im}(|\varphi|)$ with the image of $\ZZ(t)^*$ in $|\M(\kay; 6, 1)^*_{\C}|$ can be identified with the locus of  curves $C$ where the K3-surface $X(C)$ admits a ``splitting curve" of a certain degree depending on $t$, cf.  \cite{Art}, Thm.~4.6. }
\end{rem}

\begin{rem}
{\rm In \cite{DK, DGK, MSY}, occult period morphisms  are often set in comparison with 
the Deligne-Mostow theory which establishes a relation between  configuration spaces 
(e.g., of points in the projective line) and quotients of the complex unit ball by complex 
reflection groups, via monodromy groups of hypergeometric equations. This aspect of 
these examples has been suppressed entirely here. Also, it should be mentioned that there 
are other ways of constructing the period map for cubic surfaces, e.g., \cite{DK, DGK}.  }
\end{rem}
\begin{rem}\label{invar}
{\rm Let us return to the section \ref{cu}. There we had fixed an hermitian vector space $(V, (\ , \ ))$ 
over $\kay$ of signature $(n-1, 1)$. Let $V_0$ be the  underlying $\Q$-vector space, with the symmetric pairing defined by
\begin{equation*}
s( x, y)= {\rm tr}(h( x, y)) .
\end{equation*}
Then $s$ has signature $(2(n-1), 2)$, and we obtain an embedding of ${\rm U}(V)$ into ${\rm O}(V_0)$. This 
also induces an embedding of  symmetric spaces,
\begin{equation}\label{embed}
\mathcal D\hookrightarrow \mathcal D_{\rm O} ,
\end{equation}
where, as before, $\mathcal D$ is the space of negative (complex) lines in $(V_{\R}, \I_0)$, 
and where $\mathcal D_{\rm O}$ is the space of oriented negative $2$-planes in $V_{\R}$. The 
image of \eqref{embed} is precisely the set of negative $2$-planes that are stable by $\I_0$. In the 
cases of the Gauss field resp.\ the Eisenstein field, this invariance is equivalent to being stable 
under the action of $\mu_4$, resp.\ $\mu_6$. Hence in these two cases, the image of \eqref{embed} 
can also be identified with the fixed point locus of $\mu_4$ resp.\ $\mu_6$ in $\mathcal D_{\rm O}$. }
\end{rem}
\begin{rem}{\rm By going through the tables in \cite{rapo'}, \S 2, one sees that there is no further example of 
an occult period map of the type above which embeds the moduli stack of {\it hypersurfaces} of suitable degree and 
dimension into a Picard type moduli stack of abelian varieties. Note, however, that, in the case of curves of genus 4, 
the source of the hidden period morphism is a moduli stack of {\it complete intersections} of a certain multi-degree of dimension one, and there may be more examples of this type. }
\end{rem}

\bigskip
{\obeylines\parskip = 0pt
\baselineskip = 12pt

Department of Mathematics
University of Toronto
40 St. George St., BA6290
Toronto, ON M5S 2E4, Canada.
email: skudla@math.toronto.edu
}


\bigskip
{\obeylines\parskip = 0pt
\baselineskip = 12pt
Mathematisches Institut der Universit\"at Bonn  
Endenicher Allee 60 
53115 Bonn, Germany.
email: rapoport@math.uni-bonn.de
}

\end{document}